\documentclass[a4paper,10pt]{article} %bollaLapl 2025 oct.11. Random Walk incl.

\usepackage{xcolor}
\usepackage{graphicx}
\usepackage{amsfonts}
\usepackage{amstext}
\usepackage{subfigure}
\usepackage{amsthm}
\usepackage{amssymb,amsmath,amsbsy}
\usepackage[T1]{fontenc}
                
\usepackage[utf8]{inputenc}
 \usepackage{authblk}
\usepackage{caption}
\usepackage{bm}
\usepackage{comment}

\newcommand{\R}{\mathbb{R}}
\newcommand{\CC}{\mathbb{C}}
\newcommand{\E}{\mathbb{E}}
\newcommand{\0}{\mathbf{0}}
\newcommand{\1}{\mathbf{1}}

\newcommand{\x}{\mathbf{x}}

\newcommand{\uuu}{\mathbf{u}}
\newcommand{\vvv}{\mathbf{v}}
\newcommand{\w}{\mathbf{w}}
\newcommand{\y}{\mathbf{y}}
\newcommand{\Y}{\mathbf{Y}}
\newcommand{\z}{\mathbf{z}}
\newcommand{\Z}{\mathbf{Z}}
\newcommand{\I}{\bm{I}}

\newcommand{\LL}{\bm{L}}

\newcommand{\K}{\bm{K}}
\newcommand{\U}{\bm{U}}
\newcommand{\V}{\bm{V}}

\newcommand{\W}{\bm{W}}
\newcommand{\A}{\bm{A}}
\newcommand{\B}{\bm{B}}
\newcommand{\C}{\bm{C}}

\newcommand{\DD}{\bm{D}}
\newcommand{\PP}{\bm{P}}

\newcommand{\RR}{\bm{R}}

\newcommand{\T}{\bm{T}}
\newcommand{\ep}{\varepsilon}

\newcommand{\Lada}{\bm{\Lambda}}

\newcommand{\tr}{\mathrm{tr}}
\newcommand{\diag}{\mathrm{diag}}

\newcommand{\Prob}{\mathrm{Prob}\,}
\newtheorem{theorem}{Theorem}
\newtheorem{proposition}{Proposition}

\newtheorem{remark}{Remark}

\newtheorem{definition}{Definition}

\title{The non-backtracking random walk and its usage
  for vertex and edge clustering}

\author{Marianna Bolla \thanks{Department of Stochastics, Budapest University
 of Technology and Economics, M{\H u}\-egyetem rkp. 3, Budapest 1111, Hungary.
 E-mail: marib@math.bme.hu} 
}

\begin{document}

\maketitle

\section*{Abstract}

In case of sparse graphs,
relation between the real eigenvalues of the non-backtracking matrix 
and those of the  non-backtracking transition probability matrix
is considered with respect to vertex clustering.
For this purpose, the random walk along the non-backtracking graph is
considered, the vertices of which are the bioriented edges, and the adjacency
relation depends on whether the random walk goes through the oriented edges
with the rule of ``not going back in the next step''.
This is encoded in the non-backtracking matrix that is the
adjacency matrix of the non-backtracking graph.
The structural real eigenvalues of the transition probability matrix are
related to
the constant multiples of the non-backtracking one, the concordance of which
indicates the existence of a sparse stochastic block model behind the graph. 
``Inflation--deflation'' techniques are also developed for clustering
the vertices of the original
graph together with real world applications.

\vskip0.2cm  
\noindent
\textbf{Keywords}: 
non-backtracking  transition probability matrix, Bauer--Fike perturbations,
sparse stochastic block model, k-means clustering, random walk concept.

\vskip0.2cm
\noindent
\textbf{Mathematics Subject Classification}: 05C50, 05C80, 62H30

\section{Introduction}\label{intro}

Spectral graph theory was basically developed for
dense graphs, where spectra of different graph based matrices are
used for finding strongly connected clusters of the vertices,
see, e.g., \cite{Bolla13,Brouwer,Chung} for an overview. Now,
for sparse graphs, the
variety of these matrices is supplemented with the new generation of
non-backtracking based matrices and used for vertex clustering again.
Roughly speaking, these matrices are based on the edges of the sparse
graph, and a random walk along the artificially bioriented edges is
considered, where going back in the next step is prohibited. This
assumption helps to avoid the extreme effect of the high degree
vertices and makes the random walk Markovian on this so-called
non-backtracking graph. Note that it is not Markovian on the original graph,
since it has the memory of not going back in the next step.
To be more exact, some notation is introduced.

Let $G$
be a simple graph on $n$ vertices and $m$ edges. The graph is dense if
the average degree $c=\frac{2m}{n}$ is proportional to $n$, and it is
sparse if $c=o (n)$.
It makes sense in graph processes, with increasing $n$.
For example, in the sparse stochastic block model to be considered
(also see~\cite{Bolla25,Bordenave,Decelle}), $c$ is of
constant order. The adjacency matrix $\A$ and the diagonal degree matrix
$\DD$ (containing the vertex degrees that are the row-sums of $\A$) are
well-known. The (Kirchoff) Laplacian of $G$ is $\LL =\DD -\A$. This is a
positive semidefinite matrix, where the multiplicity of its 0 eigenvalue
shows the number of the connected components of $G$, and the eigenvectors
corresponding to the eigenvalues separated from 0 help us to minimize
multiway cuts or partition cuts (latter penalizes too different cluster
sizes) of the underlying dense graph. The spectrum of the
transition probability matrix $\DD^{-1} \A$ is more useful in the
random walk situation, which contains the conditional probabilities
from going to a vertex, given that the walk is at another one. Its
eigenvalues are in the $[-1,1]$ interval; the multiplicity of 1 is the
number of connected components of $G$, and -1 is an eigenvalue if and only if
$G$ is bipartite. The random walk itself is generated by the random walk
Laplacian $\I - \DD^{-1} \A$, whose eigenvalues are 1 minus the eigenvalues of
the former one, with the same eigenvectors. Its eigenvalues are also the same as
those of the normalized Laplacian $\LL_{\DD} = \DD^{-1/2} \LL \DD^{-1/2}$,
the spectrum of which is used to estimate the so-called normalized cuts and the
Cheeger constant, see~\cite{Bolla13,Chung} for details.

The multiplicity of the zero eigenvalue of $\LL$ gives the number of the
connected components of $G$. By Sylvester's inertia theorem,
this number is the same as the multiplicity
of the zero eigenvalue of the normalized Laplacian, the eigenvalues of
which are the same as those of the random walk Laplacian. Further,
the multiplicity of this zero eigenvalue is the same as the multiplicity
of the eigenvalue 1 of the transition probability matrix. In the sequel,
we assume that this multiplicity is 1, so our graph is connected.

In the connected case, the eigenvalues of the so-called normalized
modularity matrix (see~\cite{Bolla13}) differ from those of the normalized
$\I -\LL_{\DD}$ only in that 1 is not an eigenvalue of it, but there is an
additional eigenvalue 0. We proved in~\cite{Bolla15} that the largest
eigenvalue of the normalized modularity matrix is positive if and only if
our graph is not complete multipartite (including the complete graph).
Equivalently, the second largest eigenvalue (excluding the trivial 1)
of the transition probability matrix is positive.

Assume that $G$ is connected.
The random walk on $G$ forms a Markov chain that has a unique stationary
distribution of positive entries  if and only if it is ergodic,
which means irreducibility
and aperiodicity. The sufficient and necessary condition for irreducibility
that the second largest eigenvalue of the transition probability matrix is
positive, i.e., $G$ is nor complete, neither complete multipartite.
The sufficient and necessary condition for aperiodicity that the smallest
eigenvalue of the transition probability matrix is greater than -1, i.e.,
$G$ is not bipartite.

The non-backtracking matrix (see,
e.g.,~\cite{Glover21,Moore,Mossel,Newman22,Torresetal}) is a $2m\times 2m$
matrix, the vertices of which are the bioriented edges of $G$ (considered in
both possible directions), and the adjacency relation detects whether it is
possible to proceed along two connected edges, provided that going back in the
next step is prohibited. This matrix $\B$ is usually not symmetric, but it is
considered as the 0/1 adjacency matrix of the so constructed non-backtracking
graph. In the sparse stochastic block model, it has some outstanding real
eigenvalues, and the corresponding eigenvectors will be used for clustering,
after transforming and shrinking them into $n$-dimensional real vectors.

The non-backtracking graph $\cal G$ has $\B$ as adjacency matrix.
We can initiate a random walk along $\cal G$.
The non-backtracking transition probability matrix is also introduced as
${\cal T} ={\cal D}_{row}^{-1} \B$, where the diagonal matrix
${\cal D}_{row}$ contains the row-sums of $\B$ in its main diagonal.
${\cal T}$  is spectrally equivalent to the non-backtracking 
Laplacian $\I_{2m} -{\cal T}$ (see~\cite{Jost,Mulas}) that generates the
random walk along the non-backtracking graph. The spectrum of ${\cal T}$ is in
the $[-1 ,1]$ interval,
 where the eigenvectors corresponding to the outstanding real eigenvalues
 (separated from 0) are used for clustering.
 This random walk is already Markovian; on the contrary, the non-backtracking
 random walk on the original graph is not Markovian (it has the memory of
 not turning back in the next step).
$\cal T$ is also a doubly stochastic matrix that leads to a uniform
stationary distribution which is unique  if and only if
the Markov chain is ergodic.
This holds for our connected graph $G$ if it is not bipartite or complete
multipartite (including the complete graph);
equivalently, the non-backtracking graph $\cal G$ is such, provided the minimum
vertex degree of $G$ is at least 2 and it is not the cycle graph
(but the cycle graph cannot be bipartite or complete multipartite).
% The mixing rate depends on $\lambda_2$.

In~\cite{Mulas}, the notion of a \textit{circularly k-partite graph} is defined.
For $k=2$, this is a bipartite graph $G$, and $G$ is bipartite if and only if
the corresponding $\cal G$ is bipartite,
because of vertex-degree considerations,
see~\cite{Jost}. Further, if $G$ is circularly $k$-partite, then $\cal G$ is
$k$-partite, but usually not vice versa, see~\cite{Mulas}.
The fact, that $\cal T$ has a positive eigenvalue (next to the trivial 1)
is equivalent to that $\cal G$ is not complete $k$-partite with any $k$.
So, $G$ cannot be circularly $k$-partite with any $k$.

As for the Markovian random walk on the non-backtracking graph,
the walk is recurrent if the transition
probability matrix ($\cal T$) is irreducible (this is the case if our graph is
connected, itself not a cycle and the node degrees are at least 2,
exactly when $\B$ is irreducible).
In this case, starting from any vertex (which is now a directed edge)
you can reach itself or any other with positive probability in finite
(albeit large) number of steps.
In this case, the powers of $\cal T$ show the necessary
number of steps: this is the smallest $N$ for which the diagonal of
${\cal T}^N$ is strictly positive.
Our $\B$, and so, $\cal T$, contains a lot of zeros, so this $N$ can be large.

Since our $\cal T$ is also doubly stochastic, the process is ergodic with a
uniform stationary distribution. This means that starting from any directed
edge, after `many' steps, we reach any directed edge with the same probability.
On the contrary, when we consider the Markovian random walk on the original
graph $G$, the stationary distribution follows the vertex degrees, i.e., the
probability of reaching vertex $i$ is $\frac{d_i}{2m}$.
\textbf{Therefore the random
walk on $G$ concentrates on the high degree vertices, while the random
walk on $\cal G$ does not do it, it can reach any directed edge with the same
probability, due to the non-backtracking rule. This is another point to
use $\cal G$ instead of $G$.}
 
The organization of the paper as follows. In Section~\ref{nonb}, the
spectral properties of the matrices $\B$ and $\cal T$ are discussed, based on
the relevant literature (without proofs), but some new statements are also
introduced (with proofs). With the help of these facts, in Section~\ref{rel},
relation between the real eigenvalues of $\B$ and $\cal T$ is
considered. In Section~\ref{clust},  
``inflation--deflation'' techniques are  developed for clustering
the vertices of the original
graph when it comes from the sparse stochastic block model
of~\cite{Bordenave,Decelle}.
%Application to a real world graphs is shown in Section~\ref{appl}.
The non-backtracking random walk and edge clustering is discussed in
Section~\ref{edgeclust}.
The sparse stochastic block model is described in more details in
Section~\ref{sbm}.

\section{Non-backtracking  graph and transition probability matrix}\label{nonb}

The  \textit{non-backtracking matrix}
$\B =(b_{ef} )$ of a simple graph $G$ on $n$ vertices and $m$ edges
is defined as a $2m \times 2m$ 
non-symmetric matrix of 0/1 entries
(see, e.g.,~\cite{Decelle,Glover21,Jost,Krzakala}:  
$$
 b_{ef}= \delta_{e\to f} \delta_{f \ne e^{-1}} ,
$$
for $e,f \in E^{\rightarrow }$, where $E^{\rightarrow }$ is the set of bioriented
edges of $G$ (each existing edge is considered in both possible directions),
and for $e=[ i,j ]$ the reversely oriented edge is denoted $e^{-1}$, so
$e^{-1} = [ j ,i ]$; further, the $e\to f$ relation means that the endpoint
of $e$ is the starting point of $f$, denoted by $\textrm{out} (e) =
\textrm{in} (f)$;
and $\delta$ is the (1/0) indicator of the
event in its lower index. Therefore, $b_{ef} =1$ exactly when
$e\to f$ holds, but $f\ne e^{-1}$.
Since the characteristic polynomial of $\B$ has real coefficients, its
complex eigenvalues occur in conjugate pairs in the bulk of its spectrum.
Note that the underlying simple graph is not directed, just oriented, as
we consider its edges in both possible directions. 

If $G$ is connected with $d_{min} \ge 2$, then
 for all eigenvalues of $\B$, $|\mu | \ge 1$ holds.
 In particular, if $d_{min} >2$, then the
 eigenvalues of $\B$ with $|\mu |=1$ are $\pm 1$'s.
If $G$ is a connected graph that is not a cycle and $d_{\min} \ge 2$, then
$\B$ is irreducible. Therefore, the Frobenius theorem is applicable
to it, and under the above conditions, $\B$ has a single positive real
eigenvalue among its maximum absolute value ones
with corresponding eigenvector of all positive real coordinates.

 Furthermore, in case of certain random graphs (e.g., in the sparse stochastic
 block model, see~\cite{Bolla24,Decelle,Krzakala,Moore}),
 there is a bulk of the spectrum of $\B$
 (containing $\pm 1$'s and complex conjugate pairs), the other
  so-called structural eigenvalues are real, their moduli are
  greater than $\sqrt{c}$ (where $c$ is the
  average degree of the graph); they are
  positive in the assortative case, and the corresponding eigenvectors are
  nearly orthogonal, see~\cite{Bordenave}.

  Though $\B$ is not a normal matrix, it obeys the so-called PT-symmetry.
This can be described by involution and swapping.
  Introduce the notation
$$
 {\breve x}_e := x_{e^{-1}} ,  \quad  e\in  E^{\rightarrow } 
$$
for relating the coordinates of the $2m$-dimensional vectors
$\x$ and $\breve \x$  of $\CC^{2m}$. Now $\x$ is
partitioned  into two $m$-dimensional vectors $\x^{(1)}$ and $\x^{(2)}$, where
the coordinates of $\x^{(1)}$ correspond to the $j\to i$ edges with $j<i$
and those of $\x^{(2)}$ correspond to their inverses.
Then $\breve \x$ is obtained  
by swapping the first $m$ and second $m$ coordinates of $\x$.
Let $\V$ denote the following involution %on $\R^{2m}$
($\V =\V^{T}$,
$\V^2 =\I$, $\V$ is an orthogonal and symmetric matrix at the same time):
\begin{equation}\label{ve}
 \V =\begin{pmatrix}
 \bm O  & \I_m \\
 \I_m  & \bm O
\end{pmatrix} ,
\end{equation}
where the blocks are of size $m\times m$. With it, 
$\V \x = {\breve \x}$ and, vice versa,  $\V {\breve \x }= \x$.

Under non-backtracking graph $\cal G$ of $G$ we understand the directed
graph on
$2m$ vertices with adjacency relation corresponding to the definition of the
non-backtracking matrix. 
The adjacency matrix of the non-backtracking graph is just $\B$.
Historically, it is $\B^T$ that is called non-backtracking matrix,
but its eigenvalues are the same as those of $\B$.
Further, $\B^T \V = \V \B$, and $\B^T = \V \B \V$. 
This phenomenon is called PT (parity-time) invariance is physics.
 Also, $(\B \V )^T =\V \B^T =\V \V \B \V =\B \V$,
 so $\B \V$  and $\V \B$ are 
 symmetric matrices that also follows by Proposition 1 of~\cite{Bolla25}.
In the sequel, the spectral decomposition of these
symmetric matrices is described, which implies the singular value
decomposition (SVD) of $\B$.
However, the eigenvalues of $\B$ are quite different.

Now the row-sums of $\B$ are
put in the $2m\times 2m$ diagonal matrix ${\cal D}_{row}$;
here calligraphic letter is
used so that to distinguish from the diagonal degree-matrix $\DD =
\diag (d_1 ,\dots ,d_n)$ of the original graph $G$.
The diagonal entries of the row-sums of $\B^T$, or equivalently, those of the
column-sums of $\B$ are contained in
the diagonal matrix ${\cal D}_{col}$; by the PT-invariance,
${\cal D}_{col} =\V {\cal D}_{row} \V$;
so  $\diag ( {\cal D}_{row} )$ and   $\diag ( {\cal D}_{col} )$ are swappings of
each other. Trivially, the diagonal entry
of ${\cal D}_{row}$,  corresponding to the oriented edge $[i,j]$ is  $d_j -1$.
Since it has multiplicity $d_j$,
the number of edges in the non-backtracking graph is
$\sum_{j=1}^n d_j (d_j -1 ) = \sum_{j=1}^n d_j^2 -2m$; also see~\cite{Jost}. 

Note that the non-backtracking random walk on the original graph is not
Markovian (has the memory that there is no way back in the next step), but it is
Markovian on the graph of the directed edges with transition probability
matrix ${\cal T} :={\cal D}_{row}^{-1} \B$.
It means that the probability of going from the oriented edge $e$ to the
oriented edge $f$ is
$\Prob (e \to f) = \frac{1}{d_e } b_{ef}$ .
It is 0 if $f=e^{-1}$ or if the end-vertex of $e$ is not the start-vertex of $f$;
otherwise, it is $\frac{1}{d_e }$, where $d_e$ is the diagonal entry of 
${\cal D}_{row}$, corresponding to the
oriented edge $e =[i,j]$, i.e., it is $d_j -1$.
This random walk on the oriented edges is already Markovian, as a
forbidden transition
corresponds to a 0 entry of $\B$, so it has 0 probability.

At the same time, the matrix generating the random walk
through the oriented edges is the non-backtracking Laplacian
${\cal L} =\I_{2m} -{\cal T}$ of~\cite{Jost,Mulas}. 
Its eigenvalues are 1 minus the eigenvalues of $\cal T$ with the same
eigenvectors.
It is known that 0 is an eigenvalue of $\B$ if and only if $G$ contains
vertices of degree one (for example, if it is a tree).
Therefore, if $d_{min} \ge 2$,
then 1 cannot be an eigenvalue of $\cal L$, and the spectral gap of the
eigenvalues of $\cal L$ from 1 is investigated in~\cite{Jost}, and
it is bounded from below by $\frac1{d_{max} -1}$ and proved that this
  bound is sharp.
  The parity time
  symmetry also holds for $\cal T$, i.e., ${\cal T} \V$ and $\V {\cal T}$
  are symmetric matrices; further, ${\cal T}^T =\V {\cal T} \V$.

  In~\cite{Bauer} it is proved that the eigenvalues of $\cal L$ are
contained in the complex disc of center 1 and radius 1. In particular, its
real eigenvalues are in the [0,2] interval. Analogously, the eigenvalues of
$\cal T$ are in the complex disc of center 0 and radius 1;
the real ones are in the $[-1 ,1]$ interval.
The number 2 is an eigenvalue of $\cal L$, or
equivalently -1 is an eigenvalue of $\cal T$
if and only if the underlying simple graph is bipartite. Furthermore,
0 is always an eigenvalue of $\cal L$, and its multiplicity is equal to
the number of
the connected components of $\cal G$, which is
the same as the number of the connected components of $G$, whenever
none of them is the cycle graph.
Observe that $\cal L$ rather resembles the random walk Laplacian, which is
spectrally equivalent to the normalized Laplacian.

In the present, asymmetric situation, there are complex eigenvalues
(in conjugate pairs) of $\cal T$ too and
we distinguish between  right and left eigenvectors as follows.

\begin{definition}[Right and left eigenvectors]
  Let the matrix $\A$ be $n\times n$ with possible
  complex entries. %but diagonalizable and irreducible.
  The vector $\uuu \in \CC^n$ is a right eigenvector of $\A$ with eigenvalue
  $\lambda \in \CC$ if $\A\uuu =\lambda\uuu$.
  The vector $\vvv \in \CC^n$ is a left eigenvector of $\A$ with the same
  eigenvalue $\lambda$ if
  $\vvv^* \A =\lambda\vvv^*$ (equivalently, $\A^* \vvv ={\bar \lambda} \vvv$).
  If $\A$ is diagonalizable, then
  $$
   \A = \sum_{i=1}^n \lambda_i \uuu_i \vvv_i^* =\U \Lada \U^{-1} ,
  $$
  where $\U =(\uuu_1 ,\dots ,\uuu_n )$ column-wise and $\U^{-1}$ contains the
  vectors $\vvv_i^*$ row-wise.
\end{definition}\label{lr}
  Note that the matrices $\uuu_i \vvv_i^*$ in the above dyadic decomposition
  are (usually skew) projections (idempotent) as the right- and left
  eigenvectors form
  a biorthonormal system: $\vvv_i^* \uuu_j  = \uuu_j^* \vvv_i =\delta_{ij}$ and
   $\vvv_i^* \A \uuu_j =\lambda_i \delta_{ij}$ for $i,j=1,\dots ,n$.

\begin{theorem}\label{th1}
  Assume that $\B$ is irreducible and diagonalizable. Then
  the eigenvalues of ${\cal T} ={\cal D}_{row}^{-1} \B$ are allocated within
  the closed circle of
  center $\0$ and radius 1 of the complex plane $\CC$, and 1 is a single
  real eigenvalue. Furthermore, the right eigenvectors $\z_i$'s corresponding
  to the real eigenvalues $\lambda_i$ $(i=1,\dots ,k)$ of $\cal T$
  can be normalized so that they form a ${\cal D}_{row}$-orthonormal system:
  $\z_i^T {\cal D}_{row} \z_j =\delta_{ij}$ for $i,j =1,\dots ,k$.
  Further, with this normalization of $\z_i$'s,
  $\lambda_i \textrm{sign} (\lambda_i ) = \| \z_i \|^2$
  and $\w_i =- \textrm{sign} (\lambda_i ) \frac1{\lambda_i } {\breve \z_i}$
  is the corresponding left eigenvector
  of $\cal T$  for which $\z_1 ,\dots ,\z_k  $ and $\w_1 ,\dots ,\w_k$ form
  a biorthonormal system: $\z_i^T \w_j =\delta_{ij}$ for $i,j =1,\dots ,k$.
\end{theorem}
Note that we use the notation $\z_i^T$ as the eigenvectors,
corresponding to real eigenvalues of a matrix of real entries,
also have real coordinates.
However, for complex vectors and matrices, the notation $^*$ is used for
transposition after complex conjugation.
Also note that the condition $d_{min} \ge 2$
ensures that $\B$ and so, $\cal T$ are diagonalizable.

Before going to the proof, we collect some facts and statements about the
eigenvalues and left and right eigenvectors of $\cal T$, including the complex
ones too.

\begin{proposition}
  The transition probability matrix $\cal T$ is a doubly stochastic matrix.
\end{proposition}
\begin{proof}
$\cal T$ is clearly a stochastic matrix as
its row-sums are 1's. For the same reason, its transpose 
\begin{equation}\label{fordit}
({\cal D}_{row}^{-1} \B )^T = \B^T {\cal D}_{row}^{-1} ={\cal D}_{col}^{-1} \B^T
\end{equation}
is also a stochastic matrix (it is the transition probability matrix of the
reversed random walk along to the inverses of the oriented edges).
\end{proof}

Consequently, the largest modulus real eigenvalue of both $\cal T$ and
${\cal T}^T$ is 1
with eigenvector $\1$. This also means that the stationary distribution
of the corresponding Markov chain is uniform, see~\cite{Lovasz}.
Further, the uniform one is the unique stationary distribution if and only if
the Markov chain is ergodic. The necessary and sufficient condition for
ergodicity is the irreducibility ($\lambda_2>0$) and aperiodicity
($\lambda_{2m}>-1$). These hold whenever $\cal G$ is connected and not
bipartite; equivalently, $G$ is such and not the
cycle graph. %The mixing rate depends on $\lambda_2$.

The left and right eigenvectors, corresponding to the (same) real
eigenvalues of $\cal T$, form a biorthogonal
system as the matrix $\cal T$ is diagonalizable.
Summarizing, 
1 is a single real eigenvalue of the irreducible matrix
${\cal D}_{row}^{-1} \B$  of
nonnegative entries, by the Frobenius theorem; also, the moduli of the other
(possibly complex) eigenvalues are at most 1. 
A right eigenvector $\z$ with eigenvalue $\lambda$ of $\cal T$ satisfies
the equation
\begin{equation}\label{z}
{\cal D}_{row}^{-1} \B \z =\lambda \z .
\end{equation}

We will use the special structure of $\B$ when we consider a  real
eigenvalue $\lambda$ with corresponding right eigenvector $\z$ of $\cal T$.
By $\V^2 =\I_{2m}$, $\V^T =\V$ of Equation~\eqref{ve},
Equation~\eqref{z} is equivalent to
$$
 (\V {\cal D}_{row}^{-1} \V ) (\V \B \V ) (\V \z ) =\lambda (\V \z ) ,
 $$
 so  ${\cal D}_{col}^{-1} \B^T {\breve \z } = \lambda {\breve \z }$.
  Consequently, if $\z$ is a right eigenvector of ${\cal D}_{row}^{-1} \B$
  with the real eigenvalue
  $\lambda$, then $\breve \z $ is a right  eigenvector of
  ${\cal D}_{col}^{-1} \B^T$, with the same real eigenvalue $\lambda$;
  and, by Equation~\eqref{fordit},  it is also a left eigenvector of
  ${\cal D}_{row}^{-1} \B$ with the same  eigenvalue.
  
 \begin{proposition}  
  If $\z$ is a right eigenvector of $\cal T$
  with eigenvalue $\lambda \in \CC$, then  $\bar {\breve \z}$ is
  a left eigenvector of $\cal T$ with the same eigenvalue.
\end{proposition}
\begin{proof}
  Indeed,
  $$
  \begin{aligned}
  {\cal T}^* {\bar {\breve \z}} & ={\cal T}^T {\bar {\breve \z}} =
  {\cal D}_{col}^{-1} \B^T {\bar {\breve \z}} =
  (\V {\cal D}_{row}^{-1} \V ) (\V \B \V )  {\breve {\bar \z}}=
  \V ({\cal D}_{row}^{-1} \B {\bar \z}) \\
  & =\V \overline{{\cal D}_{row}^{-1} \B \z } = \V \overline {\lambda \z }  =
  \V {\bar \lambda}{\bar \z} = {\bar \lambda} (\V {\bar \z } ) =
  {\bar \lambda} {\breve {\bar \z} } =  {\bar \lambda} {\bar {\breve \z} } ,
  \end{aligned}
  $$
  so $\bar {\breve \z}$ is a right eigenvector of ${\cal T}^*={\cal T}^T$
  with eigenvalue  $\bar \lambda$, 
  which means that  $\bar {\breve \z}$ is a left
  eigenvector of $\cal T$ with eigenvalue $\lambda$.
\end{proof}

We can summarize the eigen-structures of $\cal T$ and ${\cal T}^T $ in Tables~\ref{tab1} and~\ref{tab2}.

Consequently, for complex (but not real) $\lambda$,
    the right eigenvector corresponding to
    $\lambda$ is orthogonal to the left eigenvector corresponding to
    $\bar \lambda $ and vice versa,
    the left eigenvector corresponding to
    $\lambda$ is orthogonal to the right eigenvector corresponding to
    $\bar \lambda$. That is,
    $$
    \z^* {\breve \z} =  0 \quad \textrm{and} \quad
    {\bar \z}^* {\breve{\bar \z}} = \z^T {\bar {\breve \z}} =0.
    $$  
    This follows by the biorthogonality of the right and left
    eigenvectors.
(Orthogonality holds for right-left ones corresponding to two different
eigenvalues; in particulr, for not real $\lambda$,
$\lambda \ne {\bar \lambda}$.)
With entrywise calculations, this is also proved in~\cite{Jost,Mulas},
akin to the following proposition. We cite it without proof.

\begin{table}[h]
	\centering
	\begin{tabular}{lcc}
		\hline
		${\cal T}$: \textrm{ eigenvalue} & $\lambda \in \mathbb{R}$ & $\lambda \in \mathbb{C},\ \overline{\lambda} \ne \lambda$ \\
		\hline
		right eigenvector & $\mathbf{z} \in \mathbb{R}^{2m}$ & 
		$\mathbf{z} \in \mathbb{C}^{2m},\ \bar{\mathbf{z}} \in \mathbb{C}^{2m}$ \\
		
		left eigenvector & $c\,\breve{\mathbf{z}} \in \mathbb{R}^{2m}$ &
		$c\,\breve{\bar{\mathbf{z}}}=c\,\bar{\breve{\mathbf{z}}},\ 
		\bar{c}\,\breve{\mathbf{z}} \in \mathbb{C}^{2m}$ \\
		
		& $(c \in \mathbb{R})$ & $(c \in \mathbb{C})$ \\
		\hline
	\end{tabular}
	\caption{Eigenvalues and eigenvectors of the matrix $\cal T$}\label{tab1}
\end{table}

\begin{table}[h]
	\centering
	\begin{tabular}{lcc}
		\hline
		${\cal T}^{T}$: \textrm{ eigenvalue} & $\lambda \in \mathbb{R}$ & $\overline{\lambda} \in \mathbb{C},\ \lambda \ne \overline{\lambda}$ \\
		\hline
		right eigenvector & $c \breve{\z} \in \mathbb{R}^{2m}$ & 
		$c \breve{\bar \z} =c\bar{\breve \z}, \quad {\bar c} \breve{\z } \in \CC^{2m}$ \\		
		left eigenvector & $\z \in \R^{2m}$ &
		$\z , \quad \quad \quad \quad {\bar \z} \in \CC^{2m}$ \\		
		& $(c \in \mathbb{R})$ & $(c \in \mathbb{C})$ \\
		\hline
	\end{tabular}
	\caption{Eigenvalues and eigenvectors of the matrix ${\cal T}^T$}\label{tab2}
\end{table}

\begin{proposition}\label{zero} 
  For the coordinates of any eigenvector $\z$ of $\cal T$
  that does not correspond to
   the trivial eigenvalue 1, the relation  $\sum_{j=1}^{2m} z_j =0$ holds.
\end{proposition}
This  means that $\z \perp \1$ (where $\1$ is the eigenvector corresponding to
  the eigenvalue 1; it is single if the graph is connected). However,
  the eigevectors of $\cal T$ are not usually orthogonal.
  Now we shall prove a more general statement which, in particular, implies
  Proposition~\ref{zero}. Namely,
  the coordinates of the eigenvectors of $\cal T$,
 corresponding to its (not-trivial) real eigenvalues,
 which belong to edges with the same end-point, also sum to 0.
 This happens because the corresponding (real) eigenvectors are within a
 special subspace of $\R^{2m}$, described as follows.

  To ease the discussion, two auxiliary matrices,
  defined, e.g., in~\cite{Bolla25} will be used:
  the $2m\times n$ \textit{end matrix} $\bm{End}$ has entries
  $end_{ei} =1$ if $i$ is the end-vertex of the (directed) edge $e$ and 0,
  otherwise;
  the $2m\times n$ \textit{start matrix} $\bm{Start}$ has entries
  $start_{ei} =1$ if $i$ is the start-vertex of the (directed) edge $e$ and 0,
  otherwise. Then for any vector $\uuu \in \R^n$ and for any edge
  $e=\{ i\to j \}$, the following holds:
  $$
  ( \bm{End}\, \uuu )_e = u_j \quad \textrm{and} \quad
  ( \bm{Start}\, \uuu )_e =u_i .
  $$
  Consequently,  $\bm{End} \, \uuu$ is the $2m$-dimensional inflated version
  of the $n$-dimensional vector $\uuu$, where the coordinate $u_j$ of $\uuu$
  is repeated as many times, as many edges have end-vertex $j$; likewise,
 in the $2m$-dimensional inflated vector $\bm{Start} \, \uuu$, the coordinate 
  $u_i$ of $\uuu$
  is repeated as many times, as many edges have start-vertex $i$.
  As each edge is considered in both possible directions,
  these numbers are the vertex-degrees $d_j$ and $d_i$, respectively.
 Note that 
$$
\bm{End}^* \, \bm{End}=\bm{Start}^* \, \bm{Start} =
 \diag (d_1 ,\dots ,d_n ) =\DD .
$$

 \begin{proposition}\label{apropos}
 The eigenvectors of the symmetric matrix $\B \V =
 \bm{End} \, \bm{End}^T -\I_{2m}$, corresponding to the eigenvalues
 $d_j -1$ $(j=1,\dots ,n )$ are the column vectors of $\bm{End}$. The other
 eigenvectors corresponding to the eigenvalue -1 (of multiplicity $2m-n$)
 form an arbitrary orthonormal system in the $(2m-n)$-dimensional
 real subspace, orthogonal to these $n$ column-vectors. Therefore a vector
 $\y$ within this subspace is characterized by the following equations:
 \begin{equation}\label{subspace}
 \sum_{e:\, out(e) =j} y_e =0 , \quad j=1,\dots ,n .
 \end{equation}
\end{proposition}

\begin{proof}
  The relation $\bm{End} \, \bm{End}^T= \I_{2m} + \B \V$ is trivial.
  Therefore, the eigenvalues of $\bm{End} \, \bm{End}^T$ are
  1+the eigenvalues of $\B \V$, with the same eigenvectors. These are the
 vertex-degrees (of the original graph, with possible multiplicities) and
 0 with multiplicity $2m -n$. Indeed, $\bm{End} \, \bm{End}^T$ is a positive
 semidefinite Gramian of rank $n$. It can easily be checked that the
 column-vectors of $\bm{End }$ are its 
 $n$ linearly independent eigenvectors, with the column-sums (that are the
 vertex-degrees) as eigenvalues; the other eigenvalues are zeros.
 
 These give $n$ eigenvalues $d_j -1$ $(j=1,\dots ,n)$ of $\B\V$ and the other
 multiple eigenvalue is $0-1=-1$ with multiplicity $2m-n$.
 Because a general vector within this subspace is orthogonal to each of the $n$
 column-vectors of $\bm{End}$, Equation~\eqref{subspace} is valid
 for them. This finishes the proof.
 \end{proof}
Observe that  
the symmetric matrices $\B^T \V = \V \B \V \V =\V \B$ and $\B\V$ have the same
eigenvalues and their eigenvectors are swappings of each other.
Also,  $\bm{Start} \, \bm{Start}^T = \I_{2m} + \B^T \V$, so similar
arguments as in Proposition~\ref{apropos} hold for them.
In particular, the eigenvectors $\y$'s of $\B^T \V$,
corresponding to the eigenvalue -1 (of multiplicity $2m-n$) are in the
subspace characterized by the equations
\begin{equation}\label{subspace1}
 \sum_{e:\, in(e) =j} y_e =0 , \quad j=1,\dots ,n .
\end{equation}

\begin{proposition}\label{sub}
 The eigenvectors of $\cal T$, corresponding to its (non-trivial)
 real eigenvalues are within the subspace~\eqref{subspace}.
\end{proposition}

\begin{proof}
 Indeed, by equation~\eqref{z}, 
 $\B \z =\lambda {\cal D}_{row} \z$, 
 where $\lambda \in \R$, and so, the coordinates of $\z$ are real numbers too.
 This means that for each $j\in \{ 1,\dots ,n \}$:
 \begin{equation}\label{Teigen}
 \sum_{f:\, out(f) =j} (\B \z)_f =\lambda (d_j -1 ) \sum_{e:\, out(e) =j} z_e .
 \end{equation}
 Taking into consideration the 0/1 structure of $\B$,
 $$
 (\B \z)_f = \sum_{e: \, out(e) =out(f), \, in (e) \ne in(f)} z_e
 $$
 and so,
 $$
 \sum_{f:\, out(f) =j} (\B \z)_f = \sum_{f:\, out(f) =j;}
 \sum_{e: \, out(e) =j, \, in (e) \ne in(f)} z_e =
 (d_j -1 ) \sum_{e:\, out(e) =j} z_e .
 $$
 Substituting into~\eqref{Teigen}, since $\lambda \ne 1$, this can hold only if
 $\sum_{e:\, out(e) =j} z_e =0$. This holds true for any $j=1,\dots ,n$.
\end{proof}
 
Similar result  may also hold for the real eigenvectors of $\B$,
unless the corresponding
eigenvalue is $d_j -1$ for each $j$, which excludes the regular graphs.
 Consequently, the number of the real eigenvalues of $\cal T$ is maximum $2m-n$
 and  the number of the real eigenvalues of $\B$ is maximum $2m-n$ too, but
 we know (see the relation to the eigenvalues of the $2n\times 2n$
 matrix $\K$ in~\cite{Bolla25}) that it is minimum $2m-2n$.
In Section~\ref{rel},
 we will show that in the special sparse multicluster model, the real
 eigenvalues of $\cal T$ are close to the (same) scalar multiples of those
 of $\B$.

 Proposition~\ref{apropos} also implies that
 $\B^T$ and $\B^{-1}$ have the same effect on
 the vectors of the above subspaces~\eqref{subspace} and~\eqref{subspace1},
 as it is illustrated now via their SVD's.
   By~\cite{Bordenave,Jost}, the spectral decomposition of the symmetric matrix
$\B^T \V$ is $\sum_{j=1}^{2m} \sigma_j \x_j \x_j^T$, where the eigenvectors
$\x_j$'s form an orthonormal basis (and they have real coordinates, as the
matrix is real symmetric, and the eigenvalues $\sigma_j$'s are all real).
It was shown (see~Proposition~\ref{apropos})
that the eigenvalues are the numbers $d_j -1$ (with multiplicity
$d_j$), for $j=1,\dots ,n$, with eigenvectors $\x_{2m-n+1}, \dots ,\x_{2m}$;
further, -1 (with multiplicity $2m -n$), i.e.,
the spectral decomposition of $\B^T \V$ is
$$
 \B^T \V = -\sum_{j=1}^{2m-n} \x_j \x_j^T + \sum_{j=2m-n+1}^{2m}
 (d_j -1 )  {\x}_j \x_j^T .
 $$
 Consequently,
 $$
 \B^T = -\sum_{j=1}^{2m-n} \x_j (\V \x_j)^T + \sum_{j=2m-n+1}^{2m}
 (d_j -1 )  {\x}_j (\V \x_j)^T .
 $$
Note that any linear combination of $\x_1 ,\dots ,\x_{2m-n}$ is an
 eigenvector of $\B^T \V$ with eigenvalue -1, only their subspace,
 which is just~\eqref{subspace1}, is unique.
 
This implies the SVD of $\B^T$ which is
$\B^T = \sum_{j=1}^{2m} s_j \x_j \y_j^T $, where $s_j =|\sigma_j |$ and
$\y_j = \textrm{sign} (\sigma_j ) {\breve \x}_j $. Therefore,  
the SVD of $\B^T$ is 
$$
\B^T =-\sum_{j=1}^{2m-n} \x_j {\breve{\x}_j  }^T + \sum_{j=2m-n+1}^{2m}
(d_j -1 )  \x_j {\breve{\x}_j  }^T 
$$
and the SVD of $\B$ is
\begin{equation}\label{B}
 \B =-\sum_{j=1}^{2m-n} {\breve \x}_j \x_j^T + \sum_{j=2m-n+1}^{2m}
 (d_j -1 )  {\breve \x}_j \x_j^T .  
\end{equation}
From this, with simple linear algebra, we get that the SVD of $\B^{-1}$ is  
$$
 \B^{-1} =-\sum_{j=1}^{2m-n} {\x}_j {\breve \x}_j^T + \sum_{j=2m-n+1}^{2m}
\frac1{d_j -1 }  {\x}_j {\breve \x}_j^T ,  
$$
provided each $d_j \ge 2$.

\begin{remark}\label{remi}
  As by Proposition~\ref{sub},
  the vectors $\x_1 ,\dots ,\x_{2m-n}$ span the subspace characterized
by~\eqref{subspace}, the effect of $\B^T$ and $\B^{-1}$ is the same for
vectors (namely for eigenvectors of $\cal T$ corresponding to real
eigenvalues) within this subspace.
\end{remark}

\begin{proof} (of Theorem~\ref{th1})
  The statement about the allocation of the eigenvalues is well known,
  see~\cite{Bauer,Jost}.
  
  The right eigenvalue--eigenvector equation~\eqref{z} for
  the real eigenvalues of $\cal T$ is
  equivalent to the problem of finding the generalized real eigenvalues of the
  matrices
  $\B$ and ${\cal D}_{row}$. Indeed, by Equation~\eqref{z},
 \begin{equation}\label{gen}
  \B \z_i = \lambda_i {\cal D}_{row} \z_i , \quad i=1,\dots k .
\end{equation}
 Adapting the theory of the generalized eigenvalue problem
 (see, e.g.~\cite{Rozsa}),
 since ${\cal D}_{row}$ is positive definite and $\B$ is diagonalizable,
 there exists an orthonormal system of $k$ elements with real
 coordinates which
 simultaneously diagonalizes $\B$ and ${\cal D}_{row}$ within their subspace.
 As ${\cal D}_{row}$ is
 also diagonal with positive diagonal entries, we can choose number $k$ of
  ${\cal D}_{row}$-orthonormal elements: 
  $\z_i^* {\cal D}_{row} \z_j =\z_i^T {\cal D}_{row} \z_j = \delta_{ij}$
  for $i,j=1,\dots ,k$.

  With the notation $\Z_k := (\z_1 ,\dots ,\z_{k} )$ and
  $\Lada_k := \diag (\lambda_1 ,\dots ,\lambda_{k} )$, Equation~\eqref{gen}
  can be condensed into the matrix form
  $$
  \B \Z_k = {\cal D}_{row} \Z_k \Lada_k 
  $$
 and
 \begin{equation}\label{ladak}
  \Z_k^T \B \Z_k = \Z_k^T {\cal D}_{row} \Z_k \Lada_k =\I_{k} \Lada_k =\Lada_k . 
 \end{equation}
 Therefore, the reduced rank congruent transformation with $\Z_k$ diagonalizes
 the bilinear form belonging to $\B$, at the
 same time. Again, here $\lambda_1 ,\dots ,\lambda_k$ are the real 
  eigenvalues of $\cal T$.

  Now consider the left eigenvectors $\w_i$'s of $\cal T$ that constitute a
  biorthonormal system with $\z_i$'s:
  $\z_i^* \w_j =\z_i^T \w_j =\delta_{ij}$ for $i,j =1,\dots ,k$.
  (In particular, $\w_1 \parallel {\breve \z_1 } \parallel \1$ will do,
  which is the scalar multiple of the all 1's  vector $\1$).
  In this way, we have the system of equations 
 $$
  {\cal D}_{col}^{-1} \B^T \w_i ={\lambda}_i \w_i , \quad i=1,\dots ,k , 
  $$
  because a left eigenvector of ${\cal D}_{row}^{-1} \B$ is a right eigenvector
  of its adjoint (transpose as real) ${\cal D}_{col}^{-1} \B^T$ with the
  same real eigenvalue $\lambda_i$.
    This is exactly the problem of finding the generalized eigenvalues of the
  matrices
  $\B^T$ and ${\cal D}_{col}$. Indeed,
 $$
  \B^T \w_i = {\lambda}_i {\cal D}_{col} \w_i ,\quad i=1,\dots ,k.
 $$
We know that ${\breve \Z}_k^T {\cal D}_{col} {\breve \Z}_k =\I_k$ and
  $\W_k = {\breve \Z}_k \C_k$, where  $\W_k = (\w_1 ,\dots ,\w_{k} )$,
  and $\C_k =\diag (c_1 ,\dots ,c_k )$
  is the diagonal matrix containing the real constant
  multipliers with which $\w_i =c_i {\breve \z}_i$, for $i=1,\dots ,k$,
  see also Table~\ref{tab2}.
So
$$
\W_k^* \B^T \W_k =\W_k^T \B^T \W_k = \C_k^T {\breve \Z}_k^T {\cal D}_{col}
{\breve \Z}_k \C_k {\bar \Lada}_k = \C_k \I_k \C_k \Lada_k
= \C_k^2 \Lada_k ,
$$
as $\C_k$  and $\Lada_k$ are diagonal matrices
with non-zero real diagonal entries.
Therefore, the reduced rank congruent transformation with $\W_k^T$
diagonalizes the bilinear form belonging to $\B^T$, for $i=1,\dots ,k$.
 
On the other hand, from~\eqref{ladak}  and by the biorthogonality, we have that 
   $$
   {\W}_k^T \B^{-1} {\W}_k =\Lada_k^{-1} ,
   $$
   so by the equality of the considered  part of the SVD in
   $\B^T$ and $\B^{-1}$ (see Remark~\ref{remi}),
   $\Lada_k^{-1} =\C_k^2 \Lada_k$, so $\C_k^2 ={\Lada_k}^{-2}$ and
   $c_i =\pm \frac1{\lambda_i}$ as it is a real number $(i=1,\dots ,k)$.

   Now we will show that $c_i =-\frac1{\lambda_i }$ is the correct choice
   for $\lambda_i >0$ and $c_i = \frac1{\lambda_i }$ for $\lambda_i <0$.
   Indeed, let $\z$ be a right eigenvector of $\cal T$
   corresponding to the real eigenvalue $\lambda$. It is partitioned into
   the $m$-dimensional vectors $\z^{(1)}$ and $\z^{(2)}$ the entries of which
   correspond to oriented edges in a certain ordering and to their inverses
   in the same ordering, respectively (akin to the labeling the rows/columns
   of $\B$). Due to Proposition~\ref{zero},  $\z^{(1)} +\z^{(2)} =\0$
   (the $m$-dimensional zero vector) and $\z^T {\breve \z} =
   2(\z^{(1)} )^T \z^{(2)} $. Therefore, if the corresponding left eigenvector is
   $\w =c {\breve \z}$, then by biorthonormality,
   $$
   1=\z^T \w = 2c (\z^{(1)} )^T \z^{(2)} ,
   $$
   and so,
   $$
   c = \frac1{ 2 (\z^{(1)} )^T \z^{(2)} } = - \frac1{ 2 \| \z^{(1)} \|^2 }
   =- \frac1{ \| \z \|^2 } <0 ,
   $$
   because of $\z^{(2)}  =- \z^{(1)}$.
   Since we saw that $c=\pm \frac1{\lambda}$, it follows that
   $\lambda \, \textrm{sign} (\lambda ) =\| \z \|^2$, which finishes the proof. 
\end{proof}
Note that in the assortative sparse stochastic block model, the real
eigenvalues of
   $\cal T$ are positive, like the real eigenvalues of $\B$
   (see Section~\ref{rel}). Therefore,
   $c=-\frac1{\lambda}$ and $\lambda = \| \z \|^2$ holds with the convenient
   normalization.
In the disassortative sparse stochastic block model, the real eigenvalues of
   $\cal T$ are negative, like the real eigenvalues of $\B$. Therefore,
   $c=\frac1{\lambda}$ and $\lambda = - \| \z \|^2$ holds with the convenient
   normalization.
   This argument also applies if there are both positive and negative
   real eigenvalues of $\cal T$.

 Observe that the section of the dyadic decomposition
  of $\cal T$, corresponding to its the real eigenvalues, is
 $$
 \sum_{i: \, \lambda_i \in \R} \lambda_i \z_i {\w}_i^T =
 \sum_{i: \, \lambda_i \in \R}  \lambda_i \z_i \left( -\frac{1}{\lambda_i} \right)
 {\breve \z}_i^T  =
   - \sum_{i: \, \lambda_i \in \R}  \z_i {\breve \z}_i^T .
 $$
   This also resembles the first part of the SVD of $\B^T$
  (we learned that $k<2m-n$).

\begin{remark}  
  Above, the real matrix $\B$ is considered as $\R^{2m} \to \R^{2m}$ linear
  transformation that assigns to a $2m$-dimensional real vector another
  $2m$-dimensional real one. However, due to the PT-invariance, $\B$ is a
  special one: namely, $\B \V$ is symmetric, and so,  there are the real
  orthonormal eigenvectors of $\B\V$ which appear in the SVD of $\B$ and of
  $\B^T$ and $\B^{-1}$ too as one-sided singular vectors, whereas the
  other-sided ones are negative swappings of them. They form orthonormal bases
  within the column- and row-space of the underlying matrix.
  The singular values are nonnegative real numbers.

  Now we focus on the singular value 1 of multiplicity $2m-n$. It is easy to
  see that the corresponding left and right $(2m-n)$-dimensional isotropic
  subspaces coincide, as the right one is spanned by -1 times the swappings of
  the basis elements in the left one. Therefore, this special isotropic
  subspace is an invariant subspace of $\B$, $\B^T$, and $\B^{-1}$.

  Consequently, for real vectors within this isotropic subspace, the effect of
  $\B$, $\B^T$ and $\B^{-1}$ is the same. Since the eigenvectors, corresponding
  to the real eigenvalues of $\cal T$ are within this subspace, the effect of
  $\B^T$ and $\B^{-1}$ for them is the same and will result in the same
  vector. As these eigenvectors are linearly independent,
  their number cannot exceed $(2m-n)$. From~\cite{Bolla25} we also know that
  this number is at least $2m-2n$. 
  
  However, in the spectral decomposition of $\B$ or $\cal T$, there are
  complex eigenvalues too, and the corresponding
  complex eigen-subspaces are invariant.
  Therefore, if we consider the above real matrices as
  $\CC^{2m} \to \CC^{2m}$ linear transformations, here the eigenvectors
  corresponding to the real eigenvalues of $\cal T$ are still linearly
  independent, but their number, and so, the number of the real eigenvalues
  of $\cal T$ cannot exceed $2m-n$.
\end{remark}
  
\section{Relation between the eigenvalues of $\B$ and  $\cal T$}\label{rel}  

Our ultimate goal is to find clusters of the vertices by means of the
structural non-backtracking eigenvalues that are real ones, separated from
the bulk of the spectrum (they are positive in assortative networks).
To conclude for them, it is more convenient and customary  to consider the
non-backtracking Laplacian ($\cal L$) eigenvalues separated from 1,
or equivalently, the eigenvalues of the
transition probability matrix ($\cal T$), separated from 0. All these
eigenvalues are confined to a circle of radius 1 in the complex plane,
but we are interested only in the ``structural'' (outstanding) real ones.
For this purpose, we consider the following equivalent version of
Equation~\eqref{z}:
\begin{equation}\label{zz}
  ({\cal D}_{row}^{-1/2} \B {\cal D}_{row}^{-1/2} ) ({\cal D}_{row}^{1/2} \z ) =
  \lambda ( {\cal D}_{row}^{1/2}  \z ) .
\end{equation}

It is known (see~\cite{Bolla25,Bordenave})
that in the assortative sparse stochastic block
model, $\B$ has some ``structural'' positive real eigenvalues,
greater than $\sqrt{c}$; the largest one $\mu_1$ (guaranteed by the
Frobenius theorem) is of magnitude $c$, where $c$ is the average degree of
the original graph, see the forthcoming Section~\ref{clust}.

On the other hand, consider the structural real eigenvalues
$$
1=\lambda_1 > \lambda_2 \ge \dots \ge \lambda_k >0 
$$
of $\cal T$. We learned that to $\lambda_i$ an  eigenvector $\z_i$
of $\cal T$ corresponds ($\z_i$ also has real coordinates).
By Equation~\eqref{zz},
the matrix ${\cal D}_{row}^{-1/2} \B {\cal D}_{row}^{-1/2}$ has the same eigenvalues
with eigenvectors $\x = {\cal D}_{row}^{1/2} \z$. We know that  $\z_1 =\1$
and $\x_1 = {\cal D}_{row}^{1/2} \1$.
From Theorem~\ref{th1} we also know, that these eigenvectors form a
${\cal D}_{row}$-orthonormal system, i.e. $\z_i^T {\cal D}_{row} \z_j =
\delta_{ij}$ for $i,j=1,\dots ,k$. Consequently, the vectors
$\x_i =  {\cal D}_{row}^{1/2}  \z_i$ are orthonormal:
$\x_i^T  \x_j = \delta_{ij}$ for $i,j=1,\dots ,k$.
As for  the left eigenvectors, we know that for real $\lambda$,
$$
 \w^T  ({\cal D}_{row}^{-1} \B )  = \lambda \w^T ,
 $$
 and so,
 $$
 ({\cal D}_{row}^{-1/2} \w )^T  ({\cal D}_{row}^{-1/2} \B   {\cal D}_{row}^{-1/2})
   = \lambda \w^T  {\cal D}_{row}^{-1/2} = \lambda ({\cal D}_{row}^{-1/2} \w )^T  .
    $$   
    Hence,  the vectors
    ${\cal D}_{row}^{-1/2} \w_i =-\frac1{\lambda_i}
    {\cal D}_{row}^{-1/2}{\breve \z_i }$ 
    are left eigenvectors of ${\cal D}_{row}^{-1/2} \B   {\cal D}_{row}^{-1/2}$ 
    that form a biorthonormal system with the vectors $\x_i=
    {\cal D}_{row}^{1/2}  \z_i$, where $\lambda_1 \ge \dots \ge \lambda_k >0$
    are positive real eigenvalues of $\cal T$ in the assortative
    stochastic block model.

    Note that  the right eigenvectors
    of  ${\cal D}_{row}^{-1/2} \B {\cal D}_{row}^{-1/2}$
    are orthonormal themselves and as for the
   left ones, $\lambda_i {\cal D}_{col}^{1/2} \w_i$'s form an  orthonormal system.

For the relation between the eigenvalues of $\B$ and  $\cal T$,
 Bauer--Fike type perturbations will be used. %see~\cite{Bhatia,Stewart}).
 It needs the spectral condition
 number of $\U$ of a diagonalizable matrix $\U \Lada \U^{-1}$
 (see Definition~\ref{lr}), which is
  $$
  \kappa (\U ) =\| \U \| \cdot \| \U^{-1} \| =\frac{s_{max} (\U )}{s_{min} (\U )}.
  $$
 % but latter one does not help much in our case, see~\cite{Bhatia}.
Note that $\kappa(\U ) \ge 1$ and =1 if and only if $\U$ is scalar multiple of a
  unitary matrix.
 We cite a version after~\cite{BF,Bhatia,Bordenave,Stephan,Stewart}:

  \begin{proposition}
 Let $\A = \U \Lada \U^{-1}$ be diagonalizable with eigenvalues $\alpha$'s,
 and $\B$ be arbitrary  with eigenvalues $\beta$'s (both are $n\times n$
    matrices with possibly complex entries). Then for any $\beta$ there is an
    $i \in \{ 1,\dots ,n \}$ such that
    $$
      |\beta -\alpha_i | \le \kappa (\U ) \| \B - \A \| =: R.
      $$
      There can be more than one such $i$, but we can tell the following.
      Let $C_i$ be the circle centered at $\alpha_i$ with radius $R$ (in $\CC$).
      For any union of some $C_i$'s, which is disjoint of the union of the
      remaining $C_i$'s, the number of $\beta$'s within this union is equal to
      the number of $C_i$'s in the union (or equivalently,
      to the number of $\alpha$'s in
      the union). In particular, if a $C_i$ is disjoint of the other circles,
      then there is exactly one $\beta$ in it. 
   \end{proposition}
   Note that the last part of the proposition resembles the Gersgorin
   theorem,
   and it shows that the $\beta$'s and $\alpha$'s cannot be too far apart
   in the complex plane.
   
   In our case, we apply this in the following situation.
     Let $\mu_1  \ge |\mu_2 | \ge \dots \ge |\mu_k | >0$ be the structural
     eigenvalues of the non-backtracking matrix $\B$ (in the sparse
     assortative stochastic $k$-cluster model);
     $\mu_1 \approx c$ and $|\mu_{k} | \approx \sqrt{c}$
     approximately, where $c$ is the average degree ($\mu_k$ may be complex).
     The other eigenvalues are within a circle of radius $\sqrt{c}$ in $\CC$,
     see~\cite{Bordenave,Stephan}, and we will use this model.
Assume that $\cal T$ is also diagonalizable and has positive real eigenvalues
$1=\lambda_1 \ge \lambda_2 \ge \dots \ge \lambda_{k} >0$. 
%$k:= \min \{ k_1 ,k_2 \}$. 
The base matrix is the $k$-rank approximation ${\cal T}_k$
of the transition probability 
     matrix ${\cal T} = {\cal D}_{row}^{-1} \B$ and $\frac1{\mu_1} \B$ is
     considered as the perturbed one.
     
     Then estimate $R$, which is an upper bound for the
     $|\lambda_i - \frac{\mu_i}{\mu_1}| $ differences, $i=2,\dots ,k$
     (the first difference is 0).

     $$
     \| \frac1{\mu_1} \B  - {\cal T}_k \| \le
     \|  \frac1{\mu_1} \B -{\cal T}  \| + \| {\cal T} - {\cal T}_k \|
   %  =\|( \frac1{\mu_1 }\I -{\cal D}_{row}^{-1} ) \B \| +  |\lambda_{k+1}| 
 \le \| {\cal D}_{row}^{-1}-\frac1{\mu_1 }\I \| \cdot \| \B \|
 + |\lambda_{k+1}| ,
 $$
 where $\| \B \| =\max d_i -1$ (the maximal singular value of $\B$
 (see~\eqref{B}) and
 $$
 \| {\cal D}_{row}^{-1}-\frac1{\mu_1 }\I \| =
 \max_i | \frac1{d_i -1} -\frac1{\mu_1} |
 = \frac1{\min d_i  -1} -\frac1{\mu_1 } ,
 $$
 where we used that for diagonalizable matrices of nonnegative entries,
 like $\B$, for its Frobenius eigenvaue,
 $\min  d_i -1 \le \mu_1 \le \max d_i -1$ holds (see, e.g.,~\cite{Rozsa}).
 Actually, the first part of the inequality was used
 so that to cancel the absolute value.
% where $d_i$ is around $c$, but $\min d_i <c$.

Consider the spectral decomposition
$ {\cal T}_k = \Z_k \Lada_k \Z_k^{-1}$,
where the diagonal matrix $\Lada_k$ contains
$\lambda_1 ,\dots ,\lambda_k$ and zeros along its main diagonal, whereas
$\Z_k$ contains the corresponding right eigenvectors $\z_1 ,\dots ,\z_k$
in its first $k$ columns, otherwise an arbitrary (linearly independent)
set of right eigenvectors, corresponding to the eigenvalue 0 of multiplicity
$2m -k$. As discussed before, $\W_k^T :=\Z_k^{-1}$ contains the corresponding
left eigenvectors $\w_1 ,\dots \w_k$ in its first $k$ rows and a
linearly independent set of
left eigenvectors, corresponding to the multiple eigenvalue 0, in its
subsequent rows. Since ${\cal T}_k$ has all real eigenvalues ($\lambda_1 ,
\dots ,\lambda_k$ and 0's), its right  and left eigenvectors also have
real coordinates. Hence,
$$
\frac{R}{\kappa (\Z_k )}
\le \frac{\max d_i -1}{\min d_i -1} -\frac{\max d_i -1}{\mu_1} +|\lambda_{k+1}|
\le \frac{\max d_i -1}{\min d_i -1} -1 +|\lambda_{k+1}|,
$$
and by~\cite{Jost}, $|\lambda_{k+1}| \le \frac1{\sqrt{c-1}}$ (in the
assortative sparse stochastic $k$-cluster block model).
Here the last part of the inequality
 $\min  d_i -1 \le \mu_1 \le \max d_i -1$ was used.

 Now, let us consider $\kappa (\Z_k )$.  
Since  ${\cal D}_{row}^{1/2} \Z_k$ can be chosen an orthogonal matrix, 
$$
  \| \Z_k \| = \| {\cal D}_{row}^{-1/2} ({\cal D}_{row}^{1/2} {\Z_k } )  \|
  \le \| {\cal D}_{row}^{-1/2} \| \cdot \| {\cal D}_{row}^{1/2} { \Z_k } \|
  \le \frac1{(\min d_i -1 )^{1/2} } \cdot 1.
  $$
 Utilizing the relation of Theorem~\ref{th1} between the corresponding
  left and right eigenvectors, the first $k$ rows of $\W_k$ 
 contain the  swappings of the first $k$ columns of 
  $\Z_k$  multiplied with $-\frac1{\lambda_i}$'s. Therefore,
  $$
  \| \Z_k^{-1} \| =\| \W_k \| \le \max_i \frac1{|\lambda_i | }
  \frac1{(\min d_i -1 )^{1/2} }
 \le (\max_i d_i -1 ) \frac1{(\min d_i -1 )^{1/2} } , 
 $$
 as by~\cite{Mulas}, the spectral gap between the transition probability
 spectrum and zero is at least $\frac1{\max_i d_i -1}$, see also~\cite{Jost}.

A finer estimate in the sparse assortative stochastic block model  is
$$
 \| \W_k \| \le \max_{i\le k} \frac1{|\lambda_i | } \frac1{(\min d_i -1 )^{1/2} }
 \le \sqrt{c-1} \frac1{(\min d_i -1 )^{1/2} } . 
 $$
Therefore, 
 $$
  \kappa (\Z_k ) \le \frac{\max d_i -1}{\min d_i -1 }  \quad \textrm{or}
  \quad   \kappa (\Z_k ) \le \frac{\sqrt{c-1}}{\min d_i -1 } . 
 $$
  Combining the two,
  $$
  R \le  
  \frac{\max d_i -1}{\min d_i -1}   
  \left( \frac{\max d_i -1}{\min d_i -1} -1 \right) + \frac1{\min d_i -1 } .  
  $$
 The first term  is the closer to zero as $\max d_i$ is closer to 
 $\min d_i$ which is  supported by the sparse assortative stochastic
 $k$-cluster block model with the
 conditions of~\cite{Bordenave} (namely, that the clusters approximately
 have the same average degrees).
The second term is always less than 1, but decreasing with $\min d_i$. 

Summarizing, $\lambda_i$'s are close to $\frac{\mu_i}{\mu_1}$'s
$(i=1,\dots, k)$, and by~\cite{Bordenave,Decelle} it
     will guarantee that $\lambda_k$ has a constant lower bound.
 
\section{Vertex clustering}\label{clust}

Spectral clustering algorithms use the heuristic that eigenvectors corresponding
to the structural eigenvalues of a suitable matrix are applicable to the
k-means clustering. In case of the dense stochastic block model, this is
supported by Davis--Kahan type subspace perturbation theorems.
In the sparse case, with Bauer--Fike perturbations of the previous section,
similar
arguments are used for some deflated ($n$-dimensional) versions of the
$2m$-dimensional $\B$- or $\cal T$-eigenvectors.

With fixed positive integer $k$, the $k$-cluster stochastic block model,
 $SBM_k$ was considered in~\cite{Bolla25} as follows.
The $k\times k$
\textit{probability matrix} $\PP =(p_{ab})$
of the $n$-th member $G_n \in SBM_k$ of the
percolated sparse random graph sequence is
$$
 p_{ab} =\frac{c_{ab}}{n} ,
$$
where the $k \times k$
symmetric \textit{affinity matrix} $\C =(c_{ab})$ stays constant as 
$n\to\infty$. An edge between $i<j$ comes into existence,
independently of the others, with probability
$p_{ab}$ if $i\in V_a$ and $j\in V_b$, where $(V_1 ,\dots ,V_k)$ is a partition
of the node-set $V$ into $k$ disjoint clusters. Here, of course, the size of
$V$ and $V_a$'s depend on $n$, but we do not denote this dependence, akin to
the size of the adjacency matrix; however, the relative sizes
$\frac{|V_a |}{|V|}$ will be kept constant.
This will produce the upper-diagonal
part of the $n\times n$ random adjacency matrix $\A$, and $a_{ji} :=a_{ij}$.
It can be extended to the  $i=j$ case
when self-loops are allowed, or else, the diagonal entries of the
adjacency matrix are zeros.

Let $\bar \A$ denote the $n\times n$ inflated matrix of the $k\times k$ 
matrix $\PP$: ${\bar a}_{ij} =p_{ab}$ if $i\in V_a$ and $b\in V_b$.
When loops are allowed, then $\E (a_{ij} ) ={\bar a}_{ij} $ for all
$1 \le i,j \le n$.
In the loopless case, the expected adjacency matrix $\E \A$ differs from
$\bar \A$  with respect to the the main diagonal,
but the diagonal entries are negligible. 

The average degree of a real world graph on $m$ edges and $n$ vertices is
$\frac{2m}{n}$. 
The expected average 
degree of the random graph $G_n$ generated from the $SBM_k$ model is
$$
  c=  \frac1{n} \sum_{i=1}^n \sum_{j=1}^n {\bar a}_{ij} = 
  \frac1{n} \sum_{a=1}^k \sum_{b=1}^k n_a n_b p_{ab} =
\frac1{n^2} \sum_{a=1}^k \sum_{b=1}^k n_a n_b c_{ab} =
%\sum_{a=1}^k \sum_{b=1}^k r_a r_b c_{ab} =
 \sum_{a=1}^k r_a  c_a ,
$$
where $c_a =\sum_{b=1}^k r_b c_{ab}$ is the average degree of cluster $a$.
It is valid only if self-loops are allowed. Otherwise, $c_a$ and $c$
should be decreased with a term of order $\frac1{n}$, but it will not make
too much difference in the subsequent calculations.
In~\cite{Bordenave}, the case when  $c_a =c$ for all $a$ is considered.
In this case $\frac1{c} {\bar \A}$ is a stochastic matrix, and so, the
spectral radius of $\bar \A$ is $c$.

In the assortative sparse stochastic $k$-cluster model, the
leading $k$ eigenvalues of $\B$ are aligned with those of $\bar \A$, and the
eigenvectors,
corresponding to the $k$ leading eigenvalues of $\B$ are
close to the inflated versions of those of $\bar \A$ (which is a matrix of
rank $k$), with an inner product approaching 1 as 
$n\to\infty$, see~\cite{Bordenave}.
Since latter ones are step-vectors, there is an estimate for
the sum of the inner variances of the clusters, where the  k-means objective 
is calculated via the convenient vertex representatives, see~\cite{Bolla13}.

Based on~\cite{BRJ,Bolla13}, the following can be stated.
The ground space is $S=\{ 1,\dots ,k\}$ with the measure $\mu$ of vertex types 
$r_1 ,\dots ,r_k$ (diagonal of $\RR$).
The discrete kernel $\kappa$ is in the symmetric, $k\times k$
affinity matrix $\C$. The integral operator $T_{\kappa}$ works on a
function $f:\{ 1,\dots ,k\} \to \R$ as
$$
 (T_{\kappa} f )(x) =\int_S \kappa (x,y) f(y) d\mu (y) .
 $$
 The $\int$ now is a finite sum:
 $$
  (T_{\kappa} f )_i =\sum_{j=1}^k  c_{ij} f_j r_j .
 $$
 By~\cite{BRJ}, if the spectral norm $\| T_{\kappa } \| >1$, then there is
 a supercritical state, and the giant component of size $\Theta (n)$
 appears. Now this spectral norm is the spectral norm of a $k\times k$
 symmetric matrix as follows.
 
 Say, $f$ (the $i$th coordinate function of which takes the value $f_i$,
 $i=1,\dots ,k$)
 has unit norm w.r.p. the $r_1 ,\dots ,r_k$ distribution, i.e.,
 $$
 \sum_{i=1}^k r_i f_i^2 =1 ,\quad i=1,\dots k .
 $$
 For the vector $\x \in \R^k$, $x_i =\sqrt{r_i} f_i$ $(i=1,\dots ,k)$,
 the relation $\| \x \| =1$ holds, and the above integral equation becomes
$$
(T f)_i =\frac1{\sqrt{r_i}} (T \x)_i =
         \sum_{j=1}^k c_{ij} x_j \sqrt{r_j} 
$$
and so,
$$
(T \x)_i =  \sum_{j=1}^k  \sqrt{r_i} c_{ij} \sqrt{r_j}  x_j . 
$$
The spectral norm of $T$ is the largest absolute value eigenvalue of the
symmetric matrix $\RR^{1/2} \C \RR^{1/2}$, which is single by the Perron
theorem if $\C$ has all positive entries. The real eigenvalues of 
$\RR^{1/2} \C \RR^{1/2}$ are the same as those of $\RR \C$, which are
close to the eigenvalues of $\B$ in the $SBM_k$ model, when
$c_1 =\dots =c_k =c$. Then the largest eigenvalue is $c$, so the
condition that the average degree is $c>1$ guarantees the existence of
a giant component.

Otherwise, the largest eigenvalue of $\RR \C$ in the assortative
$SBM_k$ model is between $\min c_i$ and $\max c_i$, so the condition
$d_{min} >1$ is amply sufficient for the existence of a giant component.

For the structural (real) eigenvalues of $\B$ we have
$$
 \B \x =\mu \x ,
$$
where $\x$ is close to a step-vector in the assortative sparse
stochastic $k$-cluster block model.

On the other hand, if $\z$ is eigenvector of $\cal T$ with eigenvalue $\lambda$:
$$
({\cal D}_{row}^{-1/2} \B {\cal D}_{row}^{-1/2} ) ({\cal D}_{row}^{1/2} \z ) =
  \lambda ( {\cal D}_{row}^{1/2}  \z ) ,
$$
%and if the degrees are close to each other, $\z$ is closely stepwise constant.
where the vectors ${\cal D}_{row}^{1/2}  \z_i$ $(i=1,\dots ,k)$ form an
orthonormal system.

So the structural $\lambda$'s are within a constant factor of the structural
$\mu$'s, see Section~\ref{rel}. The corresponding
eigenvectors (if the structural eigenvalues are single) are continuous
functions of the matrices. As the eigenvectors $\x_i$'s of $\B$,
corresponding to
its structural eigenvalues $\mu_i$'s are close to the inflated
versions of the eigenvectors $\uuu$'s of $\bar \A$,
they are close to step-vectors
if our graph comes from the sparse $SBM_k$ model.
Therefore, between  the $\z_i$s, as
eigenvectors of the transition probability matrix ${\cal T}={\cal D}_{row}^{-1}
\B$ and the inflated eigenvectors 
of the matrix $\DD_{\bar \A}^{-1} {\bar \A}$ 
(which are step-vectors), a similar relation holds true, as the norms of
the matrices ${\cal D}_{row}^{-1}$ and $\DD_{\bar \A}^{-1}$
% are $\frac1{v}$ and $\frac1{v-1}$, respectively, which
do not depend on $n$.
Here the diagonal matrix $\DD_{\bar \A}$ contains entries $c_i$ in the $i$th
block for $i=1,\dots ,k$ (those are the average degrees of the clusters).

The structural  eigenvalues of ${\cal D}_{row}^{-1/2} \B {\cal D}_{row}^{-1/2}$ 
are also $\lambda_i$'s with orthonormal eigenvectors
${\cal D}_{row}^{1/2}  \z_i$'s 
and those are aligned with the structural eigenvalues of 
$\DD_{\bar \A}^{-1/2} {\bar \A} \DD_{\bar \A}^{-1/2}$. Also, the unit-norm 
eigenvectors ${\cal D}_{row}^{1/2}  \z_i$'s  are close to the inflated
versions of the unit-norm 
eigenvectors of this matrix, which are step-vectors, say 
$\vvv$'s (they form an orthonormal system as the matrix is symmetric).

Summarizing, we know that
$$
\left \|  \x - \frac{{\bm{End}} \, \uuu }{\| {\bm{End} } \, \uuu \| } 
\right \|^2 \le  2 - 2(1-\frac12 \ep ) = \ep ,
$$
where $\ep$ can be any small with increasing $n$.
Therefore,
$$
\left \|  {\cal D}_{row}^{1/2} \z - \frac{{\bm{End}} \, \vvv }{\| {\bm{End} } 
\, \vvv \| } \right \|^2 \le \ep' ,
$$
where  $\ep$ can be any small with increasing $n$, but the relation between
$\ep'$ and $\ep$ does not depend on $n$. Also, $\ep' \le \ep$ since
$\| {\cal D}_{row}^{-1} \| \le 1$ and $\| \DD_{\bar \A} \| \le 1$. 

Since $({\cal D}_{row}^{1/2} \z)^{in} = \bm{End}^* {\cal D}_{row}^{1/2} \z$ and
${\bm{End}}^*{\bm{End}} =\DD $ hold by~\cite{Bolla25},
$$ 
\left \| {\bm{End}}^* {\cal D}_{row}^{1/2} \z - {\bm{End}}^* 
\frac{{\bm{End}} \, \vvv }{\|     {\bm{End} }  \, \vvv \|} \right \|^2  =
 \left \|  ({\cal D}_{row}^{1/2}\z)^{in} - \DD \frac{\vvv }{\| {\bm{End} } \, 
\vvv \|} \right \|^2 
$$
also holds.
Consequently,
$$ 
 \left \| \DD^{-1} ( {\cal D}_{row}^{1/2} \z )^{in} - \frac{\vvv }{\| 
{\bm{End} } \, 
\vvv \|} \right \|^2 \le \| \DD^{-1} {\bm{End}}^* \|^2 \ep' \le \ep' .
$$
Indeed, the largest eigenvalue of  
$(\DD^{-1} {\bm{End}}^* ) ({\bm{End}}\,\DD^{-1} ) = \DD^{-1}\DD \DD^{-1} =\DD^{-1}$
is $\max_i \frac1{d_i}$, so the largest singular value (spectral norm) of
$\DD^{-1} {\bm{End}}^*$ is estimated from above with 
$\left( \max_i \frac1{d_i} \right)^{\frac12}$.
Therefore,
$$
 \| \DD^{-1} {\bm{End}}^* \|^2 \le \max_i \frac1{d_i} = \frac1{\min_i d_i} \le 1.
$$

Now we apply this to the $k$ leading normalized eigenvectors 
$\z_1 ,\dots ,\z_k$ of $\cal T$, for which
$$
 \sum_{j=1}^k \| \DD^{-1} ({\cal D}_{row}^{1/2}\z_j)^{in} - \frac {\vvv_j }{\| 
      {\bm End } \, \vvv_j \|} \|^2 \le k \ep' .
$$
As $\vvv_j$'s are step-vectors with $k$ different coordinates on the same
$k$ steps, the above sum of the squares estimates from above
the objective function of the
k-means  algorithm. Without knowing the $\vvv_j$'s, we minimize it with
the $k$-dimensional vertex representatives. 

Akin to the calculations of \cite{Bolla25},
the weighted $k$-variance with the $k$-dimensional vertex representatives, 
obtained by the row vectors of the $n\times k$ matrix of column vectors
$\DD^{-1} ({\cal D}_{row}^{1/2}\z_j)^{in} $ $(j=1,\dots ,k )$,
will be small $(\le \ep' )$. The vectors 
${\cal D}_{row}^{1/2} \z_j$ are orthonormal by Theorem 1, and 
obtainable by numerical algorithms.

\section*{Acknowledgements}

The author thanks Reittu Hannu for preparing the Figures and the HaQuPrA
project `Harnessing Quantum for Practical Applications'.

\end{document}